\documentclass[11pt,reqno]{amsart}

\usepackage{amssymb}
\usepackage{latexcad}

\author{P\v{r}emysl Jedli\v{c}ka}
\address[Jedli\v{c}ka]{Department of Mathematics, Faculty of Engineering, Czech University of Life Sciences,
Kam\'yck\'a 129, 165 21 Prague 6–-Suchdol, Czech Republic}
\email[Jedli\v{c}ka]{jedlickap@tf.czu.cz}
\thanks{P\v{r}emysl Jedli\v{c}ka supported by the Grant Agency of the Czech Republic, grant no.~201/07/P015.}

\author{Michael K.~Kinyon}
\address[Kinyon]{Department of Mathematics, University of Denver, 2360 S Gaylord St, Denver, Colorado, 80208, U.S.A.}
\email[Kinyon]{mkinyon@math.du.edu}

\author{Petr Vojt\v{e}chovsk\'y}
\address[Vojt\v{e}chovsk\'y]{Department of Mathematics, University of Denver, 2360 S Gaylord St, Denver, Colorado, 80208, U.S.A.}
\email[Vojt\v{e}chovsk\'y]{petr@math.du.edu}

\title{Constructions of commutative automorphic loops}

\newtheorem{theorem}{Theorem}[section]
\theoremstyle{plain}

\newtheorem{conjecture}[theorem]{Conjecture}

\newtheorem{corollary}[theorem]{Corollary}

\newtheorem{example}[theorem]{Example}

\newtheorem{lemma}[theorem]{Lemma}

\newtheorem{problem}[theorem]{Problem}
\newtheorem{proposition}[theorem]{Proposition}
\newtheorem{remark}[theorem]{Remark}

\numberwithin{equation}{section}

\def\lnuc#1{N_\lambda(#1)}
\def\mnuc#1{N_\mu(#1)}
\def\rnuc#1{N_\rho(#1)}

\def\gf#1{\mathrm{GF}(#1)}
\def\mlt#1{\mathrm{Mlt}(#1)}
\def\inn#1{\mathrm{Inn}(#1)}
\def\ld{\,\backslash\,}

\def\ov#1{\overline{#1}}
\def\aut#1{\mathrm{Aut}(#1)}

\def\im{\mathop{\mathrm{Im}}}

\def\ter#1{\mathcal Q(#1)} 
\def\terg#1#2#3{\mathcal Q_{#2,#3}(#1)}

\begin{document}

\begin{abstract}
A loop whose inner mappings are automorphisms is an \emph{automorphic loop} (or
\emph{A-loop}). We characterize commutative (A-)loops with middle nucleus of
index $2$ and solve the isomorphism problem. Using this characterization and
certain central extensions based on trilinear forms, we construct several
classes of commutative A-loops of order a power of $2$. We initiate the
classification of commutative A-loops of small orders and also of order $p^3$,
where $p$ is a prime.
\end{abstract}



\maketitle

\section{Introduction}

A \emph{loop} is a groupoid $(Q,\cdot)$ with neutral element $1$ such that all
left translations $L_x:Q\to Q$, $y\mapsto xy$ and all right translations
$R_x:Q\to Q$, $y\mapsto yx$ are bijections of $Q$. Given a loop $Q$ and $x$,
$y\in Q$, we denote by $x\ld y$ the unique element of $Q$ satisfying $x(x\ld y)
= y$. In other words, $x\ld y = L_x^{-1}(y)$.

To reduce the number of parentheses, we adopt the following convention for term
evaluation: $\ld$ is less binding than juxtaposition, and $\cdot$ is less
binding than $\ld$. For instance $xy\ld u\cdot v\ld w$ is parsed as $((xy)\ld
u)(v\ld w)$.

The \emph{inner mapping group} $\inn{Q}$ of a loop $Q$ is the permutation group
generated by
\begin{displaymath}
    L_{x,y} = L_{yx}^{-1}L_yL_x,\quad R_{x,y} = R_{xy}^{-1}R_yR_x,\quad T_x = L_x^{-1}R_x,
\end{displaymath}
where $x$, $y\in Q$. A subloop of $Q$ is \emph{normal} if it is invariant under
all inner mappings of $Q$.

A loop $Q$ is an \emph{automorphic loop} (or \emph{A-loop}) if
$\inn{Q}\le\aut{Q}$, that is, if every inner mapping of $Q$ is an automorphism
of $Q$. Hence a commutative loop is an A-loop if and only if all its left inner
mappings $L_{y,x}$ are automorphisms, which can be expressed by the identity
\begin{equation}\label{Eq:A}
    xy\ld x(yu)\cdot xy\ld x(yv) = xy\ld x(y\cdot uv).\tag{\textsc{A}}
\end{equation}
Note that the class of commutative A-loops contains commutative groups and
commutative Moufang loops.

We assume that the reader is familiar with the terminology and notation of loop
theory, cf. \cite{Bruck} or \cite{Pflugfelder}. This paper is a companion to
\cite{JKV}, where we have presented a historical introduction and many new
structural results concerning commutative $A$-loops, including:
\begin{enumerate}
\item[$\bullet$] commutative A-loops are power-associative (see already
    \cite{BP}),
\item[$\bullet$] for a prime $p$, a finite commutative A-loop $Q$ has order
    a power of $p$ if and only if every element of $Q$ has order a power of
    $p$,
\item[$\bullet$] every finite commutative A-loop is a direct product of a
    loop of odd order (consisting of elements of odd order) and a loop of
    order a power of $2$,
\item[$\bullet$] commutative A-loops of odd order are solvable,
\item[$\bullet$] the Lagrange and Cauchy theorems hold for commutative
    A-loops,
\item[$\bullet$] every finite commutative A-loop has Hall $\pi$-subloops
    (and hence Sylow $p$-subloops),
\item[$\bullet$] if there is a nonassociative finite simple commutative
    A-loop, it is of exponent $2$.
\end{enumerate}
Despite these deep results, the theory of commutative A-loops is in its
infancy. As an illustration of this fact, the present theory is not
sufficiently developed to classify commutative A-loops of order $8$ without the
aid of a computer, commutative A-loops of order $pq$ (where $p<q$ are primes),
nor commutative A-loops of order $p^3$ (where $p$ is an odd prime).

The two main problems for commutative A-loops stated in \cite{JKV} were:
\emph{For an odd prime $p$, is every commutative A-loop of order $p^k$
centrally nilpotent?} \emph{Is there a nonassociative finite simple commutative
A-loop, necessarily of exponent $2$ and order a power of $2$?} For an example
of a commutative A-loop of order $8$ that is not centrally nilpotent, see
Subsection \ref{Ss:8}.

In the meantime, we have managed to solve the first problem of \cite{JKV} in
the affirmative, but we neither use nor prove the result here---it will appear
elsewhere. The second problem remains open and the many constructions of
commutative A-loops of exponent $2$ obtained here can be seen as a step toward
solving it.

One of the most important concepts in the investigation of commutative A-loops
appears to be the middle nucleus $N_\mu(Q)$, since, by \cite{BP},
$N_\lambda(Q)\le N_\mu(Q)$, $N_\rho(Q)\le N_\mu(Q)$ and $N_\mu(Q)\unlhd Q$ is
true in any A-loop $Q$. In \S\ref{Sc:Index2} we characterize all commutative
loops with middle nucleus of index $2$, solve the isomorphism problem, and then
characterize all commutative A-loops with middle nucleus of index $2$. In
\S\ref{Sc:AppsIndex2} we classify commutative A-loops of order $8$, among other
applications of \S\ref{Sc:Index2}.

Central extensions of commutative A-loops are described in
\S\ref{Sc:Extensions}. A broad class of such extensions is obtained from
trilinear forms that are symmetric with respect to an interchange of (fixed)
two arguments. As an application, we characterize all parameters $(k,\ell)$
with the property that there is a nonassociative commutative A-loop of order
$2^k$ with middle nucleus of order $2^\ell>1$.

\S\ref{Sc:p3} uses another class of central extensions partially based on the
overflow in modular arithmetic that yields many (conjecturally, all)
nonassociative commutative A-loops of order $p^3$, where $p$ is an odd prime.

A classification of commutative A-loops of small orders based on the theory and
computer computations can be found in \S\ref{Sc:Enumeration}.

\section{Commutative loops with middle nucleus of index $2$}\label{Sc:Index2}

Throughout this section, we denote by $\ov{X} = \{\ov{x};\;x \in X\}$ a
disjoint copy of the set $X$.

Let $G$ be a commutative group and $f$ a bijection of $G$. Then $G(f)$ will
denote the groupoid $(G\cup \ov{G},*)$ with multiplication
\begin{equation}\label{Eq:Gf}
    x*y = xy,\quad x*\ov{y} = \ov{xy},\quad \ov{x}*y=\ov{xy},\quad
    \ov{x}*\ov{y}=f(xy),
\end{equation}
for $x$, $y\in G$. Note that $G(f)$ is a loop with neutral element $1$.

\begin{lemma}\label{Lm:PropertiesGf}
Let $G$ be a commutative group, $f$ a bijection of $G$ and $(Q,\cdot) = G(f) =
(G\cup \ov{G},*)$. Then:
\begin{enumerate}
\item[(i)] $Q$ is commutative.
\item[(ii)] $x\ld y=x^{-1}y$, $x\ld\ov{y}=\ov{x^{-1}y}$, $\ov{x}\ld y =
\ov{x^{-1}f^{-1}(y)}$, $\ov{x}\ld\ov{y} = x^{-1}y$ for every $x$, $y\in G$.
\item[(iii)] $G\le\mnuc{Q}$.
\item[(iv)] $Q$ is a group if and only if $f$ is a translation of the group
$G$.
\item[(v)] $N_\lambda(Q)\cap G = N_\rho(Q)\cap G = Z(Q)\cap G = \{x\in
    G;\;f(xy)=xf(y)\text{ for every }y\in G\}$. When $Q$ is not a group
    $($that is, $G=\mnuc{Q})$, then $N_\lambda(Q)=N_\rho(Q)=Z(Q)\le G$.
\end{enumerate}
\end{lemma}
\begin{proof}
Part (i) follows from the definition of $G(f)$. Part (ii) is straightforward,
for instance, $x*\ov{x^{-1}y} = \ov{xx^{-1}y}=\ov{y}$ shows that
$x\ld\ov{y}=\ov{x^{-1}y}$.

For (iii), let $x$, $y$, $z\in G$ and verify that
\begin{align*}
    &x*(y*z) = (x*y)*z,\\
    &\ov{x}*(y*z)=\ov{x}*yz = \ov{xyz} = \ov{xy}*z = (\ov{x}*y)*z,\\
    &x*(y*\ov{z}) = x*\ov{yz} = \ov{xyz} = xy*\ov{z} = (x*y)*\ov{z},\\
    &\ov{x}*(y*\ov{z}) = \ov{x}*\ov{yz} = f(xyz) = \ov{xy}*\ov{z} = (\ov{x}*y)*\ov{z}.
\end{align*}
This shows $G\le \mnuc{Q}$.

(iv) An easy calculation shows that $\ov{1}\in \mnuc{Q}$ (that is, $Q$ is a
group) if and only if $f(xy)=xf(y)=f(x)y$ for every $x$, $y\in G$. With $y=1$
we deduce that $f(x)=xf(1)$ for every $x$. On the other hand, if $f(x)=xf(1)$
for every $x$ then $f(xy)=xf(y)=f(x)y$.

(v) We have $x*(y*z) = (x*y)*z$, $x*(\ov{y}*z) = \ov{xyz} = (x*\ov{y})*z$,
$x*(y*\ov{z}) = \ov{xyz} = (x*y)*\ov{z}$, and $x*(\ov{y}*\ov{z}) = xf(yz)$,
while $(x*\ov{y})*\ov{z}= f(xyz)$. Hence $x\in\lnuc{Q}$ if and only if
$xf(yz)=f(xyz)$ for every $y$, $z\in G$, which holds if and only if $xf(y) =
f(xy)$ for every $y\in G$. By commutativity, $\lnuc{Q}=\rnuc{Q}$. By (iii),
$\lnuc{Q}\cap G = Z(Q)\cap G$.

Assume that $Q$ is not a group. Suppose that $\ov{x}\in\lnuc{Q}$. Then $f(xyz)
= \ov{x} *\ov{yz} = \ov{x}*(\ov{y} *z) = (\ov{x}*\ov{y})*z = f(xy)*z= f(xy)z$
for every $y$, $z\in G$, and hence (with $y=x^{-1}$), $f(z)=f(1)z$ for every
$z\in G$. By (iv), $Q$ is a group, a contradiction. Thus $\lnuc{Q}\le G$.
\end{proof}

\begin{lemma} Let $Q$ be a commutative loop with subloop $G$ satisfying $G\le
\mnuc{Q}$, $[Q:G]=2$. Then $G$ is a commutative group and there exists a
bijection $f$ of $G$ such that $Q$ is isomorphic to $G(f)$.
\end{lemma}
\begin{proof}
The commutative loop $G$ is a group by $G\le\mnuc{Q}$. Denote by $\ov{1}$ a
fixed element of $Q\setminus G$, and define $\ov{x}=\ov{1}x=x\ov{1}$ for every
$x\in G$. Note that $\ov{1}$ is well-defined, $G\cap \ov{G}=\emptyset$ and
$Q=G\cup\ov{G}$. Moreover, $x\ov{y} = x\cdot y\ov{1} = xy\cdot \ov{1} =
\ov{xy}$ and $\ov{x}y = \ov{1}x\cdot y = \ov{1}\cdot xy = \ov{xy}$ for every
$x$, $y\in G$, using $G\le\mnuc{Q}$ again. Finally, if $x_1$, $y_1$, $x_2$,
$y_2\in G$ satisfy $x_1y_1=x_2y_2$ then
\begin{displaymath}
    \ov{x_1}\ov{y_1} = \ov{1}x_1\cdot y_1\ov{1} = \ov{1}(x_1\cdot y_1\ov{1})
    = \ov{1}(x_1y_1\cdot\ov{1}) = \ov{1}(x_2y_2\cdot \ov{1}) =
    \ov{x_2}\ov{y_2}.
\end{displaymath}
Thus the multiplication in the quadrant $\ov{G}\times\ov{G}$ mimics that of
$G\times G$, except that the elements are renamed according to the permutation
$f:G\to G$, $x\mapsto \ov{1}\cdot x\ov{1}$.
\end{proof}

\begin{corollary}\label{Cr:Gf}
Let $Q$ be a commutative loop possessing a subgroup of index $2$. Then
$[Q:\mnuc{Q}]\le 2$ if and only if there exists a commutative group $G$ and a
bijection $f$ of $G$ such that $Q$ is isomorphic to $G(f)=(G\cup\ov{G},*)$
defined by \eqref{Eq:Gf}.
\end{corollary}

\begin{remark}
The assumption that $Q$ possesses a subgroup of index $2$ in Corollary
$\ref{Cr:Gf}$ is needed only when $Q$ is a group.
\end{remark}

We now solve the isomorphism problem for nonassociative commutative loops with
middle nucleus of index $2$ in terms of the associated bijections:

\begin{proposition}\label{Pr:IsoGf}
Let $G$ be a commutative group and $f_1$, $f_2$ bijections of $G$ such that
$G(f_1)$, $G(f_2)$ are not groups. Then $G(f_1)\cong G(f_2)$ if and only if
there is $\psi\in\aut{G}$ such that
\begin{equation}\label{Eq:IsoGf}
    f_2^{-1}\psi f_1(x) = f_2^{-1}\psi f_1(1)\cdot \psi(x)\quad
    \text{for all $x\in G$},
\end{equation}
and $f_2^{-1}\psi f_1(1)$ is a square in $G$.
\end{proposition}
\begin{proof}
Denote by $*$ the multipication in $G(f_1)$, and by $\circ$ the multiplication
in $G(f_2)$.

Assume that $\varphi:G(f_1)\to G(f_2)$ is an isomorphism. Since $G(f_1)$,
$G(f_2)$ are not groups, $\varphi$ maps $\mnuc{G(f_1)}=G$ onto
$\mnuc{G(f_2)}=G$ by Lemma \ref{Lm:PropertiesGf}(iii), and hence $\psi =
\varphi|_G$ is a bijection of $G$. Then
\begin{displaymath}
    \psi(xy) = \varphi(xy) = \varphi(x*y) = \varphi(x)\circ\varphi(y) =
    \psi(x)\circ\psi(y) = \psi(x)\psi(y)
\end{displaymath}
for every $x$, $y\in G$, so $\psi\in\aut{G}$.

Define $\rho:G\to G$ by $\ov{\rho(x)} = \varphi(\ov{x})$. We have
\begin{displaymath}
    \ov{\rho(x)} = \varphi(\ov{x}) = \varphi(x*\ov{1}) =
    \varphi(x)\circ\varphi(\ov{1}) = \psi(x)\circ\ov{\rho(1)} =
    \ov{\psi(x)\rho(1)},
\end{displaymath}
so $\rho(x) = \rho(1)\psi(x)$ for every $x\in G$. Using this observation, we
have
\begin{displaymath}
    \psi(f_1(xy)) = \varphi(f_1(xy)) = \varphi(\ov{x}*\ov{y}) =
    \varphi(\ov{x})\circ \varphi(\ov{y}) = \ov{\rho(x)}\circ \ov{\rho(y)} =
    f_2(\rho(x)\rho(y)) = f_2(\rho(1)^2\psi(xy)).
\end{displaymath}
Equivalently, $f_2^{-1}\psi f_1(x) = \rho(1)^2\psi(x)$ for every $x\in G$. With
$x=1$, we deduce that $\rho(1)^2 = f_2^{-1}\psi f_1(1)$ is a square in $G$, and
that \eqref{Eq:IsoGf} holds.

Conversely, assume that \eqref{Eq:IsoGf} holds for some $\psi\in\aut{G}$, and
that $u^2 = f_2^{-1}\psi f_1(1)$ is a square in $G$. Define $\varphi:G(f_1)\to
G(f_2)$ by $\varphi(x)=\psi(x)$, $\varphi(\ov{x}) = \ov{u\psi(x)}$. Then
\begin{gather*}
    \varphi(x*y) = \varphi(xy) = \psi(xy) = \psi(x)\psi(y) = \psi(x)\circ\psi(y) =
    \varphi(x)\circ\varphi(y),\\
    \varphi(\ov{x}*y) = \varphi(\ov{xy}) = \ov{u\psi(xy)} =
    \ov{u\psi(x)\psi(y)} = \ov{u\psi(x)}\circ\psi(y) = \varphi(\ov{x})\circ
    \varphi(y),
\end{gather*}
and, similarly, $\varphi(x*\ov{y}) = \varphi(x)\circ\varphi(\ov{y})$ for every
$x$, $y\in G$. Finally, using \eqref{Eq:IsoGf} to obtain the third equality
below, we have
\begin{displaymath}
    \varphi(\ov{x}*\ov{y}) = \varphi(f_1(xy)) = \psi(f_1(xy)) =
    f_2(u^2\psi(xy)) = \ov{u\psi(x)}\circ \ov{u\psi(y)} =
    \varphi(\ov{x})\circ\varphi(\ov{y})
\end{displaymath}
for every $x$, $y\in G$. Thus $G(f_1)\cong G(f_2)$.
\end{proof}

We say that two bijections $f_1$, $f_2$ of $G$ are \emph{conjugate in
$\aut{G}$} if there is $\psi\in\aut{G}$ such that $f_2 = \psi f_1\psi^{-1}$.
The following specialization of Proposition \ref{Pr:IsoGf} will be useful in
the classification of commutative A-loops of order $8$.

\begin{corollary}\label{Cr:IsoGf}
Let $G$ be a commutative group, and let $f_1$, $f_2$ be bijections of $G$ such
that $G(f_1)$, $G(f_2)$ are not groups.
\begin{enumerate}
\item[(i)] If $f_1$, $f_2$ are conjugate in $\aut{G}$ then $G(f_1)\cong
    G(f_2)$.
\item[(ii)] If $f_1(1)=1=f_2(1)$ then $G(f_1)\cong G(f_2)$ if and only if
    $f_1$, $f_2$ are conjugate in $\aut{G}$.
\item[(iii)] If $f_2\in\aut{G}$, $t$ is a square in $G$ and
    $f_1(x)=f_2(x)t$ for every $x\in G$ then $G(f_1)\cong G(f_2)$.
\end{enumerate}
\end{corollary}
\begin{proof}
(i) Let $\psi\in\aut{G}$ be such that $f_2=\psi f_1\psi^{-1}$. Then
$f_2^{-1}\psi f_1 = \psi$, so $f_2^{-1}\psi f_1(1) = \psi(1) = 1$ is a square
in $G$ and \eqref{Eq:IsoGf} holds.

(ii) Assume that $G(f_1)\cong G(f_2)$. Then there is $\psi\in\aut{G}$ such that
\eqref{Eq:IsoGf} holds. Since $f_2^{-1}\psi f_1(1) = f_2^{-1}\psi(1) =
f_2^{-1}(1) = 1$, we deduce from \eqref{Eq:IsoGf} that $f_1$, $f_2$ are
conjugate in $\aut{G}$. The converse follows by (i).

(iii) Let $\psi$ be the trivial automorphism of $G$. Then $\eqref{Eq:IsoGf}$
becomes $f_2^{-1}f_1(x) = f_2^{-1}f_1(1)\cdot x$, and it is our task to check
this identity and that $f_2^{-1}f_1(1)$ is a square in $G$. Now,
$f_2^{-1}f_1(1) = f_2^{-1}(f_2(1)t) = f_2^{-1}(f_2(1))f_2^{-1}(t) =
f_2^{-1}(t)$ is a square in $G$ since $t$ is. Moreover, $f_1(1) = f_2(1)\cdot t
= t$, so $f_1(x)=f_1(1)f_2(x)$, and \eqref{Eq:IsoGf} follows upon applying
$f_2^{-1}$ to this equality.
\end{proof}

Finally, we describe all commutative A-loops with middle nucleus of index $2$.

\begin{proposition}\label{Pr:AGf}
Let $Q$ be a commutative loop possessing a subgroup of index $2$. Then the
following conditions are equivalent:
\begin{enumerate}
\item[(i)] $Q$ is an A-loop and $[Q:\mnuc{Q}]\le 2$.
\item[(ii)] $Q=G(f)$, where $G$ is a commutative group, $[Q:G]=2$, and
$f$ is a permutation of $G$ satisfying
\begin{align}
    &f(xy) = f(x)f(y)f(1)^{-1},\label{Eq:f1}\tag{$P_1$}\\
    &f(x^2) = x^2f(1),\label{Eq:f2}\tag{$P_2$}\\
    &f^2(x)^2f(x)^{-2}=f^2(1)\label{Eq:f3}\tag{$P_3$}
\end{align}
for every $x$, $y\in G$.
\item[(iii)] $Q=G(f)$, where $G$ is a commutative group, $[Q:G]=2$, and
$f$ is a permutation of $G$ satisfying \eqref{Eq:f1}, \eqref{Eq:f2} and
$f^2(1) = f(1)^2$.
\item[(iv)] $Q=G(f)$, where $G$ is a commutative group, $[Q:G]=2$,
$f(x)=g(x)t$ for every $x\in G$, $g\in\aut{G}$, $g(x^2)=x^2$ for every
$x\in G$, and $t$ is a fixed point of $g$.
\end{enumerate}
\end{proposition}
\begin{proof}
By Corollary \ref{Cr:Gf}, we can assume that $Q=G(f)=(G\cup \ov{G},*)$, where
$G\le \mnuc{Q}$ is a commutative group and $f$ is a bijection of $G$. Let us
establish the equivalence of (i) and (ii).

Denote by $\alpha(a,b,c,d)$ the $*$ version of \eqref{Eq:A}, namely
\begin{displaymath}
    (a*b)\ld (a*(b*(c*d))) = [(a*b)\ld (a*(b*c))]*[(a*b)\ld (a*(b*d))],
\end{displaymath}
where $a$, $b$, $c$, $d$ are taken from $G\cup\ov{G}$, and where $\ld$ is
understood in $(Q,*)$. With the exception of the variables $a$, $b$, $c$, $d$,
we implicitly assume that variables without bars are taken from $G$, while
variables with bars are taken from $\ov{G}$.

Then $\alpha(x,y,u,v)$ holds in $G(f)$, as the evaluation of $\alpha(x,y,u,v)$
takes place in the group $G$. Since $y\in \mnuc{Q}$, $\alpha(a,y,c,d)$ holds.
By commutativity of $*$, $\alpha(a,b,c,d)$ holds if and only if
$\alpha(a,b,d,c)$ holds. Hence it remains to investigate the identities
$\alpha(x,\ov{y},u,v)$, $\alpha(x,\ov{y},u,\ov{v})$,
$\alpha(x,\ov{y},\ov{u},\ov{v})$, $\alpha(\ov{x},\ov{y},u,v)$,
$\alpha(\ov{x},\ov{y},u,\ov{v})$, and $\alpha(\ov{x},\ov{y},\ov{u},\ov{v})$.

Straightforward calculation with \eqref{Eq:Gf} and Lemma \ref{Lm:PropertiesGf}
shows that $\alpha(\ov{x},\ov{y},u,\ov{v})$ holds if and only if
\begin{equation}\label{Eq:Case1}
    f(yuv) = f(xy)^{-1}f(xyu)f(yv).
\end{equation}
Using $x=y=1$, \eqref{Eq:Case1} reduces to \eqref{Eq:f1}. On the other hand,
\eqref{Eq:f1} already implies \eqref{Eq:Case1}, and so
$\alpha(\ov{x},\ov{y},u,\ov{v})$ is equivalent to \eqref{Eq:f1}. From now on,
we will assume that \eqref{Eq:f1} holds and denote $f(1)$ by $t$.

The identity $\alpha(x,\ov{y},\ov{u},\ov{v})$ is then equivalent to
\begin{equation}\label{Eq:Case2}
    x^{-1}t^{-1}=f(x^{-2})f(y^{-2})f(y)^2xt^{-5},
\end{equation}
and since $t=f(yy^{-1})=f(y)f(y^{-1})t^{-1}$ yields
\begin{equation}\label{Eq:fInverse}
    f(y^{-1})=f(y)^{-1}t^2,
\end{equation}
we can rewrite \eqref{Eq:Case2} as $f(x)^2=x^2t^2$, or, equivalently (using
\eqref{Eq:f1}), as \eqref{Eq:f2}.

Finally, note that \eqref{Eq:f1} and \eqref{Eq:fInverse} imply
\begin{equation}\label{Eq:ff}
    f^2(uv) = f(f(uv)) = f(f(u)f(v)t^{-1}) = f^2(u)f^2(v)f(t^{-1})t^{-2}
    = f^2(u)f^2(v)f(t)^{-1}.
\end{equation}
Using \eqref{Eq:ff} and \eqref{Eq:fInverse}, we see, after a lengthy
calculation, that the identity $\alpha(\ov{x},\ov{y},\ov{u},\ov{v})$ is
equivalent to \eqref{Eq:f3}.

We leave it to the reader to check that the identities $\alpha(x,\ov{y},u,v)$,
$\alpha(x,\ov{y},u,\ov{v})$, $\alpha(x,\ov{y},\ov{u},\ov{v})$ imply no
additional conditions on $f$ beside \eqref{Eq:f1}--\eqref{Eq:f3}, and,
conversely, that if \eqref{Eq:f1}--\eqref{Eq:f3} are satisfied then the
identities $\alpha(x,\ov{y},u,v)$, $\alpha(x,\ov{y},u,\ov{v})$,
$\alpha(x,\ov{y},\ov{u},\ov{v})$ hold.

We have proved the equivalence of (i) and (ii).

Assume that (ii) holds. With $x=1$ in \eqref{Eq:f3} we have
$f^2(1)^2f(1)^{-2}=f(t)$, or $f(t)^2t^{-2}=f(t)$, or $f(t)=t^2$, so (iii)
holds. Conversely, assume that (iii) holds. Then, $f^2(x)^2f(t)^{-1} =
f^2(x)^2t^{-2} = f(f(x))f(f(x))t^{-2} = f(f(x)f(x))t^{-1}=f(f(x)^2)t^{-1} =
f(x)^2$, which is \eqref{Eq:f3}, so (ii) holds.

Assume that (iii) holds and define $g$ by $g(x)=f(x)t^{-1}$, where $t=f(1)$.
Then $g(xy)=f(xy)t^{-1} = f(x)f(y)t^{-2} = f(x)t^{-1}f(y)t^{-1} = g(x)g(y)$ by
\eqref{Eq:f1}, $g(x^2) = f(x^2)t^{-1} = x^2$ by \eqref{Eq:f2}, and $g(t) =
f(t)t^{-1} = t$ by $f(t)=t^2$. Conversely, assume that (iv) holds,
$f(x)=g(x)t$, $g\in\aut{G}$, where $g(x^2)=x^2$ and $t$ is a fixed point of $g$
(not necessarily satisfying $t=f(1)$). Then $f(1) = g(1)t=t$, $f(xy) = g(xy)t =
g(x)g(y)t = g(x)tg(y)tt^{-1} = f(x)f(y)t^{-1}$, $f(x^2) = g(x^2)t = x^2t$, and
$f(t)=g(t)t = t^2$, proving (iii).
\end{proof}

\section{Constructions of commutative A-loops with middle nucleus of
index $2$}\label{Sc:AppsIndex2}

As an application of Proposition \ref{Pr:AGf}, we classify all commutative
A-loops of order $8$ and present a class of commutative A-loops of exponent $2$
with trivial center and middle nucleus of index $2$.

\subsection{Commutative A-loops of order $8$}\label{Ss:8}

It is not difficult to classify all commutative A-loops of order $8$ up to
isomorphism with a finite model builder, such as Mace4 \cite{Mace4}. It turns
out that there are $4$ nonassociative commutative A-loops of order $8$. All
such loops have middle nucleus of index $2$; a fact for which we do not have a
human proof. But using this fact, we can finish the classification by hand with
Proposition \ref{Pr:IsoGf}, Corollary \ref{Cr:IsoGf} and Proposition
\ref{Pr:AGf}.

\begin{lemma}\label{Lm:TranslatedCenters}
Let $G$ be a commutative loop, $1\ne g\in\aut{G}$ and $t\in G$. Let $f$ be a
bijection of $G$ defined by $f(x)=g(x)t$. Then $Z(G(f)) = Z(G(g))$ as sets, and
$Z(G(g)) = \{x\in G;\;g(x)=x\}$.
\end{lemma}
\begin{proof}
Since $g$ is not a translation of $G$, neither is $f$. Hence both $G(g)$ and
$G(f)$ are nonassociative, by Lemma \ref{Lm:PropertiesGf}(iv). By Lemma
\ref{Lm:PropertiesGf}(v), $Z(G(f)) = \{x\in G;\;f(xy)=xf(y)\text{ for every
}y\in G\} = \{x\in G;\;g(xy)t = xg(y)t\text{ for every }y\in G\} = \{x\in
G;\;g(xy)=xg(y)\text{ for every }y\in G\} = Z(G(g))$ and it is also equal to
$\{x\in G;\;g(x)=x\}$ since $g(xy)=g(x)g(y)$.
\end{proof}

Let $Q$ be a nonassociative commutative A-loop of order $8$, necessarily with a
middle nucleus of index $2$. By Proposition \ref{Pr:AGf}, $Q=G(f)$, where $G$
is a commutative group of order $4$ and $f(x)=g(x)t$ for some $g\in\aut{G}$ and
$t\in G$ such that $g(x^2)=x^2$ and $g(t)=t$.

Let $G=\mathbb Z_4=\langle a\rangle$ be the cyclic group of order $4$. The two
automorphisms of $G$ are the trivial automorphism $g=1$ and the transposition
$g=(a,a^3)$; both fix all squares of $G$. Let $g=1$ and $f(x)=g(x)t=xt$ for
some $t\in G$. Then $G(f)$ is a commutative group by Lemma
\ref{Lm:PropertiesGf}(iv). Assume that $g=(a,a^3)$. Then $G(g)$ is a
nonassociative commutative A-loop. The only nontrivial fixed point of $g$ is
$a^2$. Let $f(x) = g(x)a^2$. By Corollary \ref{Cr:IsoGf}(iii), $G(f)\cong
G(g)$.

Now let $G=\mathbb Z_2\times\mathbb Z_2 = \langle a\rangle \times \langle
b\rangle$ be the Klein group. Then $\aut{G} = \{1$, $(a,b)$, $(a,ab)$,
$(b,ab)$, $(a,b,ab)$, $(a,ab,b)\}\cong S_3$. The only square in $G$ is $1$ and
it is trivially fixed by all $g\in\aut{G}$.

If $g=1$ and $f(x)=g(x)t=xt$ for some $t\in G$, $G(g)$ is a commutative group
by Lemma \ref{Lm:PropertiesGf}(iv). Let $g_1=(a,b)$. The choices for $t$ are
$t=1$, $t=ab$. Let $f_1(x)=g_1(x)ab$. Then $G(g_1)$, $G(f_1)$ are
nonassociative commutative A-loops. Since $g_1(xx) = g_1(1) = 1$, $G(g_1)$ has
exponent $2$. Since $f_1(xx) = f_1(1) = ab$, $G(f_1)$ does not have exponent
$2$. Hence $G(g_1)\not\cong G(f_1)$.

Let $g_2=(a,ab)$, and note that the choices for $t$ are $t=1$, $t=b$. Let
$f_2(x) = g_2(x)b$. Since all transpositions of $S_3$ are conjugate in $S_3$,
$G(g_1)\cong G(g_2)$ by Corollary \ref{Cr:IsoGf}(i). Note that $f_1 = \psi^{-1}
f_2 \psi$ with $\psi = (b,ab)$. Hence $G(f_1)\cong G(f_2)$ by Corollary
\ref{Cr:IsoGf}. Similarly, no new nonassociative commutative A-loop of order
$8$ is obtained with $g_3=(b,ab)$.

Let $g_4 = (a,b,ab)$. Then $t=1$ is the only choice, and $G(g_4)$ is a
nonassociative commutative A-loop. By Lemma \ref{Lm:TranslatedCenters},
$Z(G(g_4))=1$ and $Z(G(f_1)) = Z(G(g_1))\cong \mathbb Z_2$. Thus $G(g_4)$ is a
new nonassociative commutative A-loop. Finally, let $g_5 = (a,ab,b)$. Since
$g_4$, $g_5$ are conjugate in $\aut{G}$, $G(g_4)\cong G(g_5)$ by Corollary
\ref{Cr:IsoGf}(i).

\subsection{A class of commutative A-loops of exponent $2$ with trivial
center and middle nucleus of index $2$}\label{Ss:Index2}

Let $\gf{2}$ be the two-element field and let $V$ be a vector space over
$\gf{2}$ of dimension $n\ge 2$. Let $G=(V,+)$ be the corresponding elementary
abelian $2$-group.

Let $\{e_1,\dots,e_n\}$ be a basis of $V$. Define an automorphism of $G$ by
\begin{displaymath}
    g(e_1)=e_2,\quad g(e_2)=e_3,\quad g(e_{n-1})=e_n,\quad g(e_n)=e_1+e_n.
\end{displaymath}
Since $g(x+x) = g(0) = 0 = g(x)+g(x)$, the equivalence of (i) and (iv) in
Proposition \ref{Pr:AGf} with $f=g$ shows that $Q_n = G(f)$ is a commutative
A-loop of order $2^{n+1}$ with nucleus of index at most $2$.

We claim that $g$ has no fixed points besides $0$. Indeed, for
$x=\sum_{i=1}^n\alpha_ie_i$ we have
\begin{displaymath}
    g(x) = \alpha_ne_1 + \alpha_1e_2+\cdots \alpha_{n-2}e_{n-1} +
    (\alpha_{n-1}+\alpha_n)e_n,
\end{displaymath}
so $x=g(x)$ if and only if
\begin{displaymath}
    \alpha_1=\alpha_n,\quad \alpha_2=\alpha_1,\quad
    \alpha_{n-1}=\alpha_{n-2},\quad \alpha_n=\alpha_{n-1}+\alpha_n,
\end{displaymath}
or, $\alpha_1=\cdots = \alpha_n=0$.

Thus Lemma \ref{Lm:TranslatedCenters} implies that $Z(Q_n)=1$, and
$[Q_n:\mnuc{Q_n}]=2$ follows. Finally, $x*x=x+x=0$ and $\ov{x}*\ov{x} = g(x+x)
= 0$ for every $x\in G$, so $Q_n$ has exponent two.

\section{Central extensions based on trilinear forms}\label{Sc:Extensions}

Let $Z$, $K$ be loops. We say that a loop $Q$ is an \emph{extension} of $Z$ by
$K$ if $Z\unlhd Q$ and $Q/Z\cong K$. If $Z\le Z(Q)$, the extension is said to
be \emph{central}.

It is well-known that central extensions of an abelian group $Z$ by a loop $K$
are precisely the loops $K\ltimes_\theta Z$ defined on $K\times Z$ by
\begin{displaymath}
    (x,a)(y,b) = (xy,\,ab\theta(x,y)),
\end{displaymath}
where $\theta:K\times K\to Z$ is a \emph{(loop) cocycle}, that is, a mapping
satisfying $\theta(x,1) = \theta(x,1)=1$ for every $x\in K$.

In \cite[Theorem 6.4]{BP}, Bruck and Paige described all central extensions of
an abelian group $Z$ by an A-loop $K$ resulting in an A-loop $Q$. The cocycle
identity they found is rather complicated, and despite some optimism of Bruck
and Paige, it is by no means easy to construct cocycles that conform to it.

In the commutative case, we deduce from \cite[Theorem 6.4]{BP}:

\begin{corollary}\label{Cr:CentralExtension}
Let $Z$ be an abelian group and $K$ a commutative A-loop. Let $\theta:K\times
K\to Z$ be a cocycle satisfying $\theta(x,y)=\theta(y,x)$ for every $x$, $y\in
K$ and
\begin{equation}\label{Eq:Cocycle}
    F(x,y,z)F(x',y,z)\theta(R_{y,z}(x),R_{y,z}(x')) = F(xx',y,z)\theta(x,x')
\end{equation}
for every $x$, $y$, $z$, $x'\in K$, where
\begin{displaymath}
    F(x,y,z) = \theta(R_{y,z}(x),yz)^{-1}\theta(y,z)^{-1}\theta(xy,z)\theta(x,y).
\end{displaymath}
Then $K\ltimes_\theta Z$ is a commutative A-loop.

Conversely, every commutative A-loop that is a central extension of $Z$ by $K$
can be represented in this manner.
\end{corollary}

\begin{corollary}
Let $Z$ be an elementary abelian $2$-group and $K$ a commutative A-loop of
exponent two. Let $\theta:K\times K\to Z$ be a cocycle satisfying $\theta(x,y)
= \theta(y,x)$ for every $x$, $y\in K$, $\theta(x,x)=1$ for every $x\in K$, and
\begin{multline}\label{Eq:Cocycle2}
    \theta(x,y)\theta(x',y)\theta(xx',y)\theta(x,x')\theta(xy,z)\theta(x'y,z)\theta(y,z)\theta((xx')y,z)=\\
    \theta(R_{y,z}(x),yz)\theta(R_{y,z}(x'),yz)\theta(R_{y,z}(xx'),yz)\theta(R_{y,z}(x),R_{y,z}(x'))
\end{multline}
for every $x$, $y$, $z$, $x'\in K$. Then $K\ltimes_\theta Z$ is a commutative
A-loop of exponent two.

Conversely, every commutative A-loop of exponent two that is a central
extension of $Z$ by $K$ can be represented in this manner.
\end{corollary}

When $K$ is an elementary abelian $2$-group, the cocycle identity
\eqref{Eq:Cocycle2} can be rewritten as
\begin{align}
    &\theta(x,y)\theta(x',y)\theta(xx',y)\notag\\
    &\theta(xy,z)\theta(x'y,z)\theta(xx',z)\label{Eq:CocycleLines}\\
    &\theta(x,yz)\theta(x',yz)\theta(xx',yz)\notag\\
    &\theta(y,z)\theta(xx',z)\theta((xx')y,z)=1.\notag
\end{align}
Since every line above is of the form $\theta(u,w)\theta(v,w)\theta(uv,w)$, it
is tempting to try to satisfy \eqref{Eq:Cocycle2} by imposing
$\theta(u,w)\theta(v,w)\theta(uv,w)=1$ for every $u$, $v$, $w\in K$. However,
that identity already implies associativity. A nontrivial solution to the
cocycle identity for commutative A-loops of exponent two can be obtained as
follows:

\begin{proposition}\label{Pr:Trilinear}
Let $Z=\gf{2}$ and let $K$ be an elementary abelian $2$-group. Let
$g:K^3\to\gf{2}$ be a trilinear form such that $g(x,xy,y) = g(y,xy,x)$ for
every $x$, $y\in K$. Define $\theta:K^2\to\gf{2}$ by $\theta(x,y) = g(x,xy,y)$.
Then $Q=K\ltimes_\theta Z$ is a commutative A-loop of exponent $2$. Moreover,
$(y,b)\in \mnuc{Q}$ if and only if for every $x$, $z\in K$ we have
$g(y,x,z)=g(x,z,y)$.
\end{proposition}
\begin{proof}
Trilinearity alone implies that
$\theta(u,w)\theta(v,w)\theta(uv,w)=g(u,v,w)g(v,u,w)$. The left-hand side of
\eqref{Eq:CocycleLines} can then be rewritten as
\begin{displaymath}
    g(x,x',y)g(x',x,y)g(xy,x'y,z)g(x'y,xy,z)g(x,x',yz)g(x',x,yz)g(y,xx',z)g(xx',y,z),
\end{displaymath}
which reduces to $1$ by trilinearity.

We have $(y,b)\in\mnuc{Q}$ if and only if $\theta(x,y)\theta(xy,z) =
\theta(y,z)\theta(x,yz)$ for every $x$, $z\in K$, and the rest follows from
trilinearity of $g$.
\end{proof}

Let $V=\gf{2}^n$.  Call a $3$-linear form $g:V\to\gf{2}$
\emph{$(1,3)$-symmetric} if $g(x,y,z) = g(z,y,x)$ for every $x$, $y$, $z\in V$.
By Proposition \ref{Pr:Trilinear}, a $(1,3)$-symmetric trilinear form gives
rise to a commutative A-loop $Q$ of exponent $2$, and $(y,b)\in \mnuc{Q}$ if
and only if $g(y,x,z)=g(y,z,x)$ for every $x$, $z$, that is, if and only if the
induced bilinear form $g(y,-,-):V^2\to \gf{2}$ is symmetric.

\begin{lemma}\label{Lm:NewForms}
Let $V$ be a vector space over $\gf{2}$ of dimension at least $3$. Then there
exists a trilinear form $g:V\to\gf{2}$ such that for any $0\ne x\in V$ the
induced bilinear form $g(x,-,-):V^2\to \gf{2}$ is not symmetric.
\end{lemma}
\begin{proof}
Let $\{e_1,\dots,e_n\}$ be a basis of $V$. The trilinear form $g$ is determined
by the values $g(e_i,e_j,e_k)\in\gf{2}$, for $1\le i$, $j$, $k\le n$. Set
$g(e_i,e_i,e_{i+1})=1$ for every $i$ (with $e_{n+1}=e_1$) and
$g(e_i,e_j,e_k)=0$ otherwise.

Let $x = \sum\alpha_j e_j$ be such that $\alpha_i\ne 0$ for some $i$. Then
$g(x,e_i,e_{i+1}) = \sum \alpha_j g(e_j,e_i,e_{i+1}) = \alpha_i
g(e_i,e_i,e_{i+1}) = \alpha_i\ne 0$, while, similarly, $g(x,e_{i+1},e_i)=0$.
\end{proof}

\begin{example}\label{Ex:SmallMiddleNucleus}
By Lemma $\ref{Lm:NewForms}$, for every $n\ge 3$ there is a commutative A-loop
$Q$ of exponent $2$ and order $2^{n+1}$ with $\mnuc{Q}=Z(Q)$, $|Z(Q)|=2$.
\end{example}

Let $Q$ be a finite commutative A-loop of exponent $2$. By results of
\cite{JKV}, $|Q|=2^k$ for some $k$. Let $|\mnuc{Q}|=2^\ell$. We show how to
realize all possible pairs $(k,\ell)$ with $\ell>0$.

\begin{lemma}\label{Lm:PossibleMiddleNuclei}
Let $k\ge \ell>0$. Then there is a nonassociative commutative A-loop of order
$2^k$ with middle nucleus of order $2^\ell$ if and only if: either $d=k-\ell\ge
3$, or $d\ge 1$ and $\ell\ge 2$.
\end{lemma}
\begin{proof}
If $d\ge 3$, consider the loop $Q$ of order $2^{d+1}$ with middle nucleus of
order $2$ from Example \ref{Ex:SmallMiddleNucleus}. Then $Q\times(\mathbb
Z_2)^{k-(d+1)}$ achieves the parameters $(k,\ell)$.

Assume that $d=2$. The parameters $(3,1)$ are not possible by
\S\ref{Sc:AppsIndex2}, and the parameters $(4,2)$ are possible (see
\S\ref{Sc:Enumeration}). Then $(k,\ell)$ can be achieved using the appropriate
direct product.

Finally, assume that $d=1$. Then we are done by Subsection \ref{Ss:Index2}. We
obviously must have $\ell\ge 2$, else $|Q|=2^k\le 4$.
\end{proof}

We remark that Lemma \ref{Lm:NewForms} cannot be improved:

\begin{lemma}\label{Lm:LowDimension}
Let $V=\gf{2}^n$ and let $g:V^3\to\gf{2}$ be a $(1,3)$-symmetric trilinear
form. If $n<3$ then there is $0\ne x\in V$ such that the induced form
$g(x,-,-)$ is symmetric.
\end{lemma}
\begin{proof}
There is nothing to show when $n=1$, so assume that $n=2$ and $\{e_1$, $e_2\}$
is a basis of $V$. The form $g$ is determined by the $6$ values
$g(e_1,e_1,e_1)$, $g(e_1,e_1,e_2)$, $g(e_1,e_2,e_1)$, $g(e_1,e_2,e_2)$,
$g(e_2,e_1,e_2)$ and $g(e_2,e_2,e_2)$.

Suppose that no induced form $g(x,-,-)$ is symmetric, for $0\ne x\in V$. Then
$g(e_1,e_1,e_2)\ne g(e_1,e_2,e_1)$, else $g(e_1,-,-)$ is symmetric. Similarly,
$g(e_2,e_1,e_2)\ne g(e_2,e_2,e_1)$. But then $g(e_1+e_2,e_1,e_2) =
g(e_1,e_1,e_2)+g(e_2,e_1,e_2) = g(e_1,e_2,e_1)+g(e_2,e_2,e_1) =
g(e_1+e_2,e_2,e_1)$, hence $g(e_1+e_2,-,-)$ is symmetric, a contradiction.
\end{proof}

\begin{remark}
The many examples presented so far might suggest that $Q/\mnuc{Q}$ is a group
in every commutative A-loop. This is not so: Consider a commutative Moufang
loop $Q$. Then $Q$ is a commutative A-loop, and $\mnuc{Q}=Z(Q)$ since the three
nuclei of $Q$ coincide. So the statement ``$Q/\mnuc{Q}$ is a group'' is
equivalent to ``$Q/Z(Q)$ is an abelian group'', i.e., to ``$Q$ has nilpotency
class at most $2$''. There are commutative Moufang loops of nilpotency class
$3$.
\end{remark}

\begin{problem} Find a smallest commutative A-loop $Q$ in which $Q/\mnuc{Q}$ is
not a group.
\end{problem}

\subsection{Adding group cocycles}

Let $Z$ be an abelian group and $K$ a loop. Then a loop cocycle $\theta:K\times
K\to Z$ is said to be a \emph{group cocycle} if it satisfies the identity
\begin{equation}\label{Eq:GroupCocycle}
    \theta(x,y)\theta(xy,z)=\theta(y,z)\theta(x,yz).
\end{equation}
Note that if $K$ is a group and $\theta$ is a group cocycle then
$K\ltimes_\theta Z$ is a group, too.

\begin{lemma}\label{Lm:AddGroupCocycle}
Let $Z$ be an abelian group, $K$ a group and $\theta$, $\mu:K\times K\to Z$
loop cocycles such that $\nu=\theta\mu^{-1}:(x,y)\mapsto
\theta(x,y)\mu(x,y)^{-1}$ is a group cocycle. Then the left inner mappings in
$K\ltimes_\theta Z$ and $K\ltimes_\mu Z$ coincide.
\end{lemma}
\begin{proof}
Calculating in $K\ltimes_\theta Z$, we have
\begin{align*}
    (x,a)(y,b) &= (xy,ab\theta(x,y)),\\
    (x,a)\ld (y,b) &= (x\ld y, a^{-1}b\theta(x,x\ld y)^{-1}).
\end{align*}
Then
\begin{multline}\label{Eq:SameCoset}
    (x,a)(y,b)\ld (x,a)((y,b)(z,c))
        = (xy,ab\theta(x,y))\ld (xyz,abc\theta(x,yz)\theta(y,z)\\
        = (z,c\theta(x,yz)\theta(y,z)\theta(x,y)^{-1}\theta(xy,z)^{-1}).
\end{multline}
Thus the left inner mappings in $K\ltimes_\theta Z$ and $K\ltimes_\mu Z$
coincide if and only if
\begin{displaymath}
    \theta(x,yz)\theta(y,z)\theta(x,y)^{-1}\theta(xy,z)^{-1}
    = \mu(x,yz)\mu(y,z)\mu(x,y)^{-1}\mu(xy,z)^{-1}
\end{displaymath}
for every $x$, $y$, $z\in K$, which happens precisely when $\nu =
\theta\mu^{-1}$ is a group cocycle.
\end{proof}

\begin{lemma}\label{Lm:AddGroupCocycle2}
Let $Z$ be an abelian group, $K$ a group and $\theta:K\times K\to Z$ a cocycle
such that $K\ltimes_\theta Z$ is a commutative A-loop. Let $\mu:K\times K\to Z$
be a group cocycle satisfying $\mu(x,y)=\mu(y,x)$ for every $x$, $y\in K$. Then
$K\ltimes_{\mu\theta} Z$ is a commutative A-loop with the same (left) inner
mappings as $K\ltimes_\theta Z$.
\end{lemma}
\begin{proof}
Both $Q_\theta = K\ltimes_\theta Z$, $Q_{\mu\theta}=K\ltimes_{\mu\theta} Z$ are
commutative loops. Since $\mu\theta\theta^{-1}$ is a group cocycle,
$Q_{\mu\theta}$ has the same (left) inner mappings as $Q_\theta$, by Lemma
\ref{Lm:AddGroupCocycle}. It therefore remains to show that every left inner
mapping of $Q_{\mu\theta}$ is an automorphism.

Let $(x,a)$, $(y,b)\in K\times Z$ and let $\varphi$ be a left inner mapping of
$Q_{\mu\theta}$ (and hence of $Q_\theta$). Denote by $\cdot$ the multiplication
in $Q_\theta$ and by $*$ the multiplication in $Q_{\mu\theta}$. Then
\begin{displaymath}
    \varphi((x,a)*(y,b)) = \varphi((x,a)\cdot(y,b)\cdot(1,\mu(x,y)))
    = \varphi((x,a))\cdot \varphi((y,b))\cdot (1,\mu(x,y)),
\end{displaymath}
because $(1,\mu(x,y))\in Z$ is a central element. The equation
\eqref{Eq:SameCoset} in fact shows that $\varphi((x,a)) = (x,a')$ for some
$a'$, and similarly, $\varphi((y,b)) = (y,b')$ for some $b'$. Thus
\begin{displaymath}
    \varphi((x,a))\cdot\varphi((y,b))\cdot(1,\mu(x,y))
    = (x,a')\cdot (y,b')\cdot (1,\mu(x,y))
    = (x,a')*(y,b') = \varphi((x,a))*\varphi((y,b)),
\end{displaymath}
proving $\varphi\in\aut{Q_{\mu\theta}}$.
\end{proof}

\section{A class of commutative A-loops of order $p^3$}\label{Sc:p3}

Let $Q$ be a commutative A-loop of odd order. Equivalently, let $Q$ be a finite
commutative A-loop in which the mapping $x\mapsto x^2$ is a bijection of $Q$
(cf. \cite[Lemma 3.1]{JKV}). For $x\in Q$, denote by $x^{1/2}$ the unique
element of $Q$ such that $(x^{1/2})^2 = x$. Define a new operation $\circ$ on
$Q$ by
\begin{displaymath}
    x\circ y = (x^{-1}\ld xy^2)^{1/2}.
\end{displaymath}
By \cite[Lemma 3.5]{JKV}, $(Q,\circ)$ is a Bruck loop. By \cite[Corollary
3.11]{JKV}, $(Q,\circ)$ is commutative if and only if it is isomorphic to $Q$.

\begin{proposition}\label{Pr:p2}
Let $p$ be an odd prime, and let $Q$ be a commutative
A-loop of order $p$, $2p$, $4p$, $p^2$, $2p^2$ or $4p^2$. Then $Q$ is an
abelian group.
\end{proposition}
\begin{proof}
Loops of order less than $5$ are abelian groups. By the Decomposition Theorem
mentioned in the introduction, it remains to prove that commutative A-loops of
order $p$ and $p^2$ are abelian groups. For $|Q|=p$, this follows from the
Lagrange Theorem and power-associativity. Assume that $|Q|=p^2$. Then
$(Q,\circ)$ is a Bruck loop of order $p^2$, in particular a Bol loop of order
$p^2$. Burn showed in \cite{Burn} that all Bol loops of order $p^2$ are groups,
and hence $(Q,\circ)$ is an abelian group. Consequently, $Q$ is an abelian
group.
\end{proof}

In this section we initiate the study of nonassociative commutative A-loops of
order $p^3$. We conjecture that the class of loops constructed below accounts
for all such loops.

\begin{lemma} There is no commutative A-loop with center of prime index.
\end{lemma}
\begin{proof}
For a contradiction, let $Q$ be a commutative A-loop such that $|Q/Z(Q)|=p$ for
some prime $p$. By Proposition \ref{Pr:p2}, $Q/Z(Q)$ is the cyclic group of
order $p$. Let $x\in Q\setminus Z(Q)$. Then $|xZ(Q)|=p$ and every element of
$Q$ can be written as $x^iz$, where $0\le i<p$ and $z\in Z(Q)$. With $0\le i$,
$j$, $k<p$ and $z_1$, $z_2$, $z_3\in Z(Q)$ we have
\begin{displaymath}
    (x^iz_1\cdot x^jz_2)\cdot x^kz_3 = (x^ix^j)x^k\cdot z_1z_2z_3 =
    x^i(x^jx^k)\cdot z_1z_2z_3 = x^iz_1\cdot (x^jz_2\cdot x^kz_3)
\end{displaymath}
by power-associativity, so $Q$ is an abelian group with center of prime index,
a contradiction.
\end{proof}

Hence a nonassociative commutative A-loop of order $p^3$ has center of size $1$
or $p$. (By the result announced in the introduction, we know, in fact, that
the center must have size $p$ if $p$ is odd.)

Let $n\ge 1$. The \emph{overflow indicator} is the function $(-,-)_n:\mathbb
Z_n\times\mathbb Z_n\to \{0,1\}$ defined by
\begin{displaymath}
    (x,y)_n = \left\{\begin{array}{ll}
        1,\text{ if $x+y\ge n$,}\\
        0,\text{ otherwise}.
    \end{array}\right.
\end{displaymath}
Denote by $\oplus$ the addition in $\mathbb Z_n$, and note that for $x$,
$y\in\mathbb Z_n$ we have $x\oplus y = x+y-n(x,y)_n$, and thus
\begin{equation}\label{Eq:Indicator}
    (x,y)_n = \frac{x+y-(x\oplus y)}{n}.
\end{equation}

\begin{lemma}\label{Lm:IndicatorGroup}
We have
\begin{equation}\label{Eq:IndicatorGroup}
    (x,y)_n + (x\oplus y,z)_n = (y,z)_n + (x,y\oplus z)_n
\end{equation}
for every $x$, $y$, $z\in\mathbb Z_n$.
\end{lemma}
\begin{proof}
Using \eqref{Eq:Indicator}, the identity \eqref{Eq:IndicatorGroup} can be
rewritten as
\begin{displaymath}
    x+y-(x\oplus y) + (x\oplus y)+z - (x\oplus y\oplus z) = y+z - (y\oplus
    z) + x + (y\oplus z) - (x\oplus y\oplus z),
\end{displaymath}
which holds.
\end{proof}

From now on we write $+$ for the addition in $\mathbb Z_n$, too.

For $n\ge 1$ and $a$, $b\in\mathbb Z_n$, define $\terg{\mathbb Z_n}{a}{b}$ on
$\mathbb Z_n\times\mathbb Z_n\times\mathbb Z_n$ by
\begin{equation}\label{Eq:Terg}
    (x_1,x_2,x_3)(y_1,y_2,y_3) = (x_1+y_1+(x_2+y_2)x_3y_3 + a(x_2,y_2)_n + b(x_3,y_3)_n, x_2+y_2,x_3+y_3).
\end{equation}
Then $\terg{\mathbb Z_n}{a}{b}$ can be seen as a central extension of $\mathbb
Z_n$ by $\mathbb Z_n\times\mathbb Z_n$ via the loop cocycle
$\theta((x_2,x_3),(y_2,y_3)) = (x_2+y_2)x_3y_3+a(x_2,y_2)_n+b(x_3,y_3)_n$, and
hence $\terg{\mathbb Z_n}{a}{b}$ is a commutative loop with neutral element
$(0,0,0)$.

Note that we can write $\theta$ as $\theta=\mu+\nu$, where
$\mu((x_2,y_2),(x_3,y_3)) = (x_2+y_2)x_3y_3$ and $\nu((x_2,y_2),(x_3,y_3)) =
a(x_2,y_2)_n + b(x_3,y_3)_n$. By Lemma \ref{Lm:IndicatorGroup}, $\nu$ is a
group cocycle.

\begin{proposition}\label{Pr:Terg}
Let $n\ge 2$ and $a$, $b\in\mathbb Z_n$. Let $Q=\terg{\mathbb Z_n}{a}{b}$ and
$x=(x_1,x_2,x_3)$, $y=(y_1,y_2,y_3)$, $z=(z_1,z_2,z_3)\in Q$. Then:
\begin{enumerate}
\item[(i)] $x\ld y = (y_1-x_1-(y_3-x_3)x_3y_2-a(x_2,y_2-x_2)_n -
b(x_3,y_3-x_3)_n,y_2-x_2,y_3-x_3)$,
\item[(ii)] $L_{y,x}(z) = xy\ld x(yz) = (z_1+y_3(x_3z_2-x_2z_3),z_2,z_3)$,
\item[(iii)] $Q$ is a nonassociative commutative A-loop of order $n^3$,
\item[(iv)] $\lnuc{Q} = Z(Q) = \mathbb Z_n\times 0\times 0$, $\mnuc{Q} = \mathbb
Z_n\times\mathbb Z_n\times 0$ as subsets of $Q$,
\item[(v)] $Q/Z(Q)\cong\inn{Q}\cong \mathbb Z_n\times\mathbb Z_n$, and
    $\inn{Q} = \{L_{u,v};\;u,v\in Q\}$,
\item[(vi)] for every $m\ge 0$, $x^m = (mx_1+2\binom{m+1}{3}x_2x_3^2 +
    at_2+bt_3,mx_2,mx_3)$, where $t_i = \sum_{k=1}^{m-1}(x_i,kx_i)_n$. (As
    usual, the summation is considered empty and the binomial coefficient
    vanishes when $m<2$.)
\end{enumerate}
\end{proposition}
\begin{proof}
Part (i) follows from the multiplication formula \eqref{Eq:Terg}. Let $Q_0 =
\terg{\mathbb Z_n}{0}{0}$. By Lemma \ref{Lm:AddGroupCocycle}, it suffices to
verify the formula (ii) for $Q_0$ instead of $Q$. Now, calculating in $Q_0$,
\begin{displaymath}
    x(yz) = (x_1+y_1+z_1+(y_2+z_2)y_3z_3 +
    (x_2+y_2+z_2)x_3(y_3+z_3),x_2+y_2+z_2,x_3+y_3+z_3),
\end{displaymath}
so (i) for $Q_0$ implies that $xy\ld x(yz)$ is equal to
\begin{displaymath}
    (z_1+(y_2+z_2)y_3z_3+(x_2+y_2+z_2)x_3(y_3+z_3)-(x_2+y_2)x_3y_3
    -z_3(x_3+y_3)(x_2+y_2+z_2),z_2,z_3),
\end{displaymath}
which simplifies in a straightforward way to (ii).

By Lemma \ref{Lm:AddGroupCocycle2}, to verify that left inner mappings of $Q$
are automorphisms of $Q$, it suffices to check that the left inner mappings of
$Q_0$ are automorphisms of $Q_0$. With $u=(u_1,u_2,u_3)$, $v=(v_1,v_2,v_3)$,
use (ii) to see that
\begin{align*}
    &xy\ld x(yu)\cdot xy\ld x(yv) \\
    &=(u_1+y_3(x_3u_2-x_2u_3),u_2,u_3)(v_1+y_3(x_3v_2-x_2v_3),v_2,v_3)\\
    &=(u_1{+}v_1{+}y_3(x_3(u_2{+}v_2){-}x_2(u_3{+}v_3)){+}(u_2{+}v_2)u_3v_3{+}a(u_2,v_2)_n{+}b(u_3,v_3)_n, u_2{+}v_2, u_3{+}v_3)\\
    &=xy\ld x(y\cdot uv).
\end{align*}
Hence $Q$ is a commutative A-loop of order $n^3$.

To calculate the middle nucleus, we can once again resort to the loop $Q_0$,
since the group cocycle will not play any role in identities that are
consequences of associativity. We have
\begin{align*}
    y\cdot (x_1,x_2,0)z &= y(x_1+z_1,x_2+z_2,z_3)\\
    &=(x_1+y_1+z_1+(x_2+y_2+z_2)y_3z_3,x_2+y_2+z_2,y_3+z_3)\\
    &=(x_1+y_1,x_2+y_2,y_3)z = y(x_1,x_2,0)\cdot z,
\end{align*}
so $\mathbb Z_n\times\mathbb Z_n\times 0\le \mnuc{Q_0}$. On the other hand,
\begin{displaymath}
    (0,0,x_3)(x_1,x_2,0) = (x_1,x_2,x_3),
\end{displaymath}
so to prove that $(x_1,x_2,x_3)\not\in\mnuc{Q_0}$ whenever $x_3\ne 0$, it
suffices to show that $(0,0,x_3)\not\in\mnuc{Q_0}$ whenever $x_3\ne 0$. Now,
\begin{multline*}
    (0,0,1)\cdot (0,0,x_3)(0,1,0) = (0,0,1)(0,1,x_3) = (x_3,1,1+x_3)\\
    \ne (0,1,1+x_3) = (0,0,1+x_3)(0,1,0) = (0,0,1)(0,0,x_3)\cdot (0,1,0)
\end{multline*}
shows just that. Similarly,
\begin{align*}
    (x_1,0,0)\cdot yz &= (x_1,0,0)(y_1+z_1+(y_2+z_2)y_3z_3,y_2+z_2,y_3+z_3)\\
    &=(x_1+y_1+z_1+(y_2+z_2)y_3z_3,y_2+z_2,y_3+z_3)\\
    &=(x_1+y_1,y_2,y_3)z = (x_1,0,0)y\cdot z
\end{align*}
proves that $\mathbb Z_n\times 0\times 0\le \lnuc{Q_0}$, and, for $x_2\ne 0$,
\begin{multline*}
    (x_1,x_2,0)\cdot (0,0,1)(0,0,1) = (x_1,x_2,0)(0,0,2) = (x_1,x_2,2)\\
    \ne (x_1+x_2,x_2,2) = (x_1,x_2,1)(0,0,1) = (x_1,x_2,0)(0,0,1)\cdot (0,0,1)
\end{multline*}
implies that $\lnuc{Q} = \mathbb Z_n\times 0\times 0$ (recall that
$\lnuc{Q}\le\mnuc{Q}$ in any A-loop $Q$).

Consider the mapping $\varphi:Q\to\inn{Q}$ defined by
\begin{displaymath}
    \varphi(x_1,x_2,x_3) = L_{(0,0,1),(0,x_2,x_3)}.
\end{displaymath}
Then
\begin{align*}
    &\varphi(x_1,x_2,x_3)\varphi(y_1,y_2,y_3)(z_1,z_2,z_3)\\
    &=\varphi(x_1,x_2,x_3)(z_1+y_3z_2-y_2z_3,z_2,z_3) = (z_1+y_3z_2-y_2z_3+x_3z_2-x_2z_3,z_2,z_3)\\
    &=\varphi((x_1,x_2,x_3)(y_1,y_2,y_3))(z_1,z_2,z_3)
\end{align*}
and $\varphi$ is a homomorphism. Its kernel consists of all $(x_1,x_2,x_3)\in
Q$ such that $x_3z_2-x_2z_3=0$ for every $z_2$, $z_3\in Q$. Thus $\ker{\varphi}
= \{(x_1,0,0);\;x_1\in \mathbb Z_n\}$. To prove (v), it remains to show that
$\varphi$ is onto $\inn{Q}$. By (ii),
\begin{displaymath}
    L_{(y_1,y_2,y_3),(x_1,x_2,x_3)} = L_{(0,0,y_3),(0,x_2,x_3)} =
    L_{(0,0,1),(0,y_3x_2,y_3x_3)}.
\end{displaymath}
This means that $\im\varphi$ contains a generating subset of $\inn{Q}$, and
hence it is equal to $\inn{Q}$. In fact, purely on the grounds of cardinality,
we have $\inn{Q} = \{L_{u,v};\;u,\,v\in Q\}$.

The identity of (vi) clearly holds when $m=0$. Assume that it holds for some
$m\ge 0$. Let $t_i^m = \sum_{k=1}^m (x_i,kx_i)_n$. By power-associativity, we
have
\begin{align*}
    x^{m+1} &= xx^m = x(mx_1+2\binom{m+1}{3}x_2x_3^2 + at_2^{m-1} + bt_3^{m-1}, mx_2,mx_3)\\
    &= ((m{+}1)x_1{+}2\binom{m+1}{3}x_2x_3^2{+}(m{+}1)x_2mx_3^2{+}at_2^m{+}bt_3^m, (m{+}1)x_2,(m{+}1)x_3),
\end{align*}
Since $2\binom{m+1}{3}+(m+1)m = 2\binom{m+2}{3}$, we are through.
\end{proof}

\begin{lemma}\label{Lm:IsoExp}
Let $p$ be a prime and $a$, $b\in\mathbb Z_p$. Let
$Q=\terg{\mathbb Z_p}{a}{b}$. Then:
\begin{enumerate}
\item[(i)] if $(a,b)=(0,0)$ and $p\ne 3$ then $Q$ has exponent $p$,
\item[(ii)] if $(a,b)\ne(0,0)$ or $p=3$ then $Q$ has exponent $p^2$,
\item[(iii)] if $a=0$ then $\mnuc{Q}\cong\mathbb Z_p\times\mathbb Z_p$,
\item[(iv)] if $a\ne 0$ then $\mnuc{Q}\cong\mathbb Z_{p^2}$.
\end{enumerate}
\end{lemma}
\begin{proof}
By \cite{JKV}, every element of $Q$ has order a power of $p$, so $Q$ has
exponent $p$, $p^2$ or $p^3$. Since $Q$ is nonassociative by Proposition
\ref{Pr:Terg}, the exponent is either $p$
    or $p^2$.

Assume that $(a,b)=(0,0)$. Then by Proposition \ref{Pr:Terg}(vi),
\begin{displaymath}
    (x_1,x_2,x_3)^p = (2\binom{p+1}{3}x_2x_3^2,0,0).
\end{displaymath}
The integer $2\binom{p+1}{3}$ is divisible by $p$ if and only if $p\ne 3$. This
proves (i).

To show (ii), it remains to prove that $Q$ has exponent $p^2$ if $(a,b)\ne
(0,0)$. Assume that $a\ne 0$, and note that, by Proposition \ref{Pr:Terg}(vi),
\begin{displaymath}
    (0,1,0)^p = (a\sum_{k=1}^{p-1}(1,k)_p,0,0) = (a(1,p-1)_p,0,0) = (a,0,0).
\end{displaymath}
This means that $Q$ does not have exponent $p$, and it also shows, by
Proposition \ref{Pr:Terg}(iv), that $\mnuc{Q}\cong\mathbb Z_{p^2}$. Similarly,
when $b\ne 0$, use
\begin{displaymath}
    (0,0,1)^p = (b\sum_{k=1}^{p-1}(1,k)_p,0,0) = (b,0,0)
\end{displaymath}
to conclude that $Q$ does not have exponent $p$.

Finally, when $a=0$, we have $(x_1,x_2,0)^p=0$ by Proposition
\ref{Pr:Terg}(vi), so $\mnuc{Q}\cong\mathbb Z_p\times\mathbb Z_p$ by
Proposition \ref{Pr:Terg}(iv).
\end{proof}

As usual, denote by $\mathbb Z_n^*$ the set of all invertible elements of
$\mathbb Z_n$.

\begin{lemma}\label{Lm:TergIso1}
Let $n>0$. If $b$, $c\in \mathbb Z_n^*$ then $\terg{\mathbb Z_n}{0}{b} \cong
\terg{\mathbb Z_n}{0}{c}$.
\end{lemma}
\begin{proof}
Define $\varphi:\terg{\mathbb Z_n}{0}{b}\to\terg{\mathbb Z_n}{0}{c}$ by
$(x_1,x_2,x_3)\mapsto ((c/b)x_1,(c/b)x_2,x_3)$, and note that $\varphi$ is a
bijection since $b$, $c$ are invertible.

Denote by $\cdot$ the multiplication in $\terg{\mathbb Z_n}{0}{b}$ and by $*$
the multiplication in $\terg{\mathbb Z_n}{0}{c}$. Then
\begin{align*}
    &\varphi((x_1,x_2,x_3)\cdot(y_1,y_2,y_3)) = \varphi((x_1+y_1+(x_2+y_2)x_3y_3+b(x_3,y_3)_n,x_2+y_2,x_3+y_3))\\
    &=(\frac{c}{b}(x_1+y_1+(x_2+y_2)x_3y_3+b(x_3,y_3)_n),\frac{c}{b}(x_2+y_2),x_3+y_3)\\
    &=(\frac{c}{b}x_1,\frac{c}{b}x_2,x_3)*(\frac{c}{b}y_1,\frac{c}{b}y_2,y_3)
    = \varphi((x_1,x_2,x_3))*\varphi((y_1,y_2,y_3)).
\end{align*}
\end{proof}

Let $p$ be an odd prime. Recall that $a\in\mathbb Z_p^*$ is a \emph{quadratic
residue modulo $p$} if there is $x\in\mathbb Z_p^*$ such that $x^2\equiv a\pmod
p$. Else $a$ is a \emph{quadratic nonresidue modulo $p$}. Also recall that
$ab^{-1}$ is a quadratic residue if and only if either both $a$, $b$ are
quadratic residues or both $a$, $b$ are quadratic nonresidues.

\begin{lemma}\label{Lm:IsoQuadraticResidue}
Let $p$ be an odd prime and $a_1$, $a_2\in\mathbb Z_p^*$. If $a_1$, $a_2$ are
either both quadratic residues or both quadratic nonresidues then
$\terg{\mathbb Z_p}{a_1}{0}\cong\terg{\mathbb Z_p}{a_2}{0}$.
\end{lemma}
\begin{proof}
Since $a_1a_2^{-1}$ is a quadratic residue, there is $u$ such that
$a_2=a_1u^2$. Define $\varphi:\terg{\mathbb Z_p}{a_1}{0}\to\terg{\mathbb
Z_p}{a_2}{0}$ by $(x_1,x_2,x_3)\mapsto (u^2x_1,x_2,ux_3)$. Then $\varphi$ is a
bijection. Denote by $\cdot$ the multiplication in $\terg{\mathbb Z_p}{a_1}{0}$
and by $*$ the multiplication in $\terg{\mathbb Z_p}{a_2}{0}$. Then
\begin{align*}
    &\varphi((x_1,x_2,x_3)\cdot (y_1,y_2,y_3)) = \varphi((x_1+y_1+(x_2+y_2)x_3y_3+a_1(x_2,y_2)_p,x_2+y_2,x_3+y_3))\\
    &= (u^2(x_1+y_1+(x_2+y_2)x_3y_3+a_1(x_2,y_2)_p),x_2+y_2,u(x_3+y_3))\\
    &= (u^2x_1+u^2y_1+(x_2+y_2)ux_3uy_3+a_2(x_2,y_2)_p, x_2+y_2, u(x_3+y_3))\\
    &= (u^2x_1,x_2,ux_3)*(u^2y_1,y_2,uy_3) = \varphi((x_1,x_2,x_3))*\varphi((y_1,y_2,y_3)).
\end{align*}
\end{proof}

\begin{lemma}\label{Lm:Isof}
For a prime $p$, let $Q_1=\terg{\mathbb Z_p}{a}{b}=(Q_1,\cdot)$,
$Q_2=\terg{\mathbb Z_p}{a}{c}=(Q_2,*)$ and let $f:Q_1\to Q_2$ be an isomorphism
that pointwise fixes the middle nucleus of $Q_1$ $($i.e., $f$ is identical on
$\mathbb Z_p\times\mathbb Z_p\times 0)$. Then there are $A$, $B\in\mathbb Z_p$
and $C\in\mathbb Z_p^*$ such that
\begin{equation}\label{Eq:Isof}
    f(x_1,x_2,x_3) = (x_1,x_2,0)*(A,B,C)^{x_3}
\end{equation}
for every $(x_1,x_2,x_3)\in Q_1$.

In addition, every mapping $f:Q_1\to Q_2$ defined by \eqref{Eq:Isof} with $A$,
$B\in\mathbb Z_p$ and $C\in \mathbb Z_p^*$ is a bijection that pointwise fixes
$\mnuc{Q_1}$.
\end{lemma}
\begin{proof}
Let $f:Q_1\to Q_2$ be an isomorphism that pointwise fixes $\mnuc{Q_1}$. As
$Q_1/\mnuc{Q_1}$ is a cyclic group, $f$ is determined by the image of any
element in $Q_1\setminus \mnuc{Q_1}$. Let $f(0,0,1)=(A,B,C)$. We must have
$C\ne 0$, else $f$ is not a bijection. Since
$(x_1,x_2,x_3)=(x_1,x_2,0)(0,0,x_3)$ and $(0,0,x_3) = (0,0,1)^{x_3}$ by
Proposition \ref{Pr:Terg}(vi), we have
\begin{displaymath}
    f(x_1,x_2,x_3) = f(x_1,x_2,0)*f(0,0,1)^{x_3} = (x_1,x_2,0)*(A,B,C)^{x_3}.
\end{displaymath}

Conversely, define $f:Q_1\to Q_2$ by \eqref{Eq:Isof}, where $C\ne 0$. Then $f$
obviously pointwise fixes $\mnuc{Q_1}$. To show that $f$ is a bijection, assume
that $f(x_1,x_2,x_3) = f(y_1,y_2,y_3)$. Since the last coordinate of
$(x_1,x_2,0)*(A,B,C)^{x_3}$ is $Cx_3$, we conclude that $x_3=y_3$. The second
coordinate of $(x_1,x_2,0)*(A,B,C)^{x_3}$ is $x_2+Bx_3$, and we conclude that
$x_2=y_2$. Then $x_1=y_1$ follows from the multiplication formula for $Q_2$ and
from Proposition \ref{Pr:Terg}(vi).
\end{proof}

\begin{lemma}\label{Lm:IsoSameA}
Let $p\ne 3$ be a prime and assume that $a$, $b$, $c\in\mathbb
Z_p$ are such that $a+c\equiv b\pmod p$. Let $Q_1=\terg{\mathbb Z_p}{a}{b} =
(Q_1,\cdot)$ and $Q_2=\terg{\mathbb Z_p}{a}{c} = (Q_2,*)$. Then $f:Q_1\to Q_2$
defined by \eqref{Eq:Isof} with $(A,B,C) = (0,1,1)$ is an isomorphism.
\end{lemma}
\begin{proof}
For $x\in \mathbb Z_p$, let $x' = (x-1)x(x+1)/3$. By Lemma \ref{Lm:Isof}, $f$
is a bijection onto $Q_2$ that pointwise fixes $\mnuc{Q_1}$. Upon expanding the
formula \eqref{Eq:Isof}, we see that
\begin{displaymath}
    f(x_1,x_2,x_3) = (x_1+x_3'+ a(x_2,x_3)_p, x_2+x_3, x_3),
\end{displaymath}
since the expression $\sum_{k=1}^{x_3-1}(1,k)_p$ vanishes for every $x_3<p$.
Let
\begin{displaymath}
    (u_1,u_2,u_3) = f(x_1,x_2,x_3)*f(y_1,y_2,y_3)
\end{displaymath}
and
\begin{displaymath}
     (v_1,v_2,v_3) = f((x_1,x_2,x_3)\cdot (y_1,y_2,y_3)).
\end{displaymath}
A quick calculation then shows that
\begin{displaymath}
    (u_2,u_3)=(v_2,v_3)=(x_2+x_3+y_2+y_3,x_3+y_3),
\end{displaymath}
$u_1$ is equal to
\begin{displaymath}
   x_1{+}x_3'{+}a(x_2,x_3)_p{+}y_1{+}y_3'{+}a(y_2,y_3)_p{+}(x_2{+}x_3{+}y_2{+}y_3)x_3y_3{+}a(x_2{+}x_3,y_2{+}y_3)_p{+}c(x_3,y_3)_p,
\end{displaymath}
while $v_1$ is equal to
\begin{displaymath}
   x_1{+}y_1{+}(x_2{+}y_2)x_3y_3{+}a(x_2,y_2)_p{+}b(x_3,y_3)_p{+}(x_3{+}y_3)'{+}a(x_2{+}y_2,x_3{+}y_3)_p.
\end{displaymath}
Now, $x_3'+y_3' = (x_2+y_2)x_3y_3 + (x_3+y_3)'$. Using \eqref{Eq:Indicator}, it
is easy to see that
\begin{displaymath}
    (x_2,x_3)_p + (y_2,y_3)_p+(x_2+x_3,y_2+y_3)_p = (x_2,y_2)_p +
    (x_2+y_2,x_3+y_3)_p + (x_3,y_3)_p.
\end{displaymath}
Hence we are done by $a+c\equiv b\pmod p$.
\end{proof}

\begin{corollary}\label{Cr:IsoExp}
Let $p\ne 3$ be a prime, $a\in\mathbb Z_p^*$ and $b$, $c\in\mathbb Z_p$. Then
$\terg{\mathbb Z_p}{a}{b}$ is isomorphic to $\terg{\mathbb Z_p}{a}{c}$.
\end{corollary}
\begin{proof}
By Lemma \ref{Lm:IsoSameA} we have $\terg{\mathbb Z_p}{a}{0}\cong \terg{\mathbb
Z_p}{a}{a}\cong\terg{\mathbb Z_p}{a}{2a}$, and so on.
\end{proof}

\subsection{Ring construction}

Note that for $a=b=0$, the construction \eqref{Eq:Terg} makes sense over any
commutative ring $R$, not just over $\mathbb Z_n$. We can summarize the most
important features of the construction as follows:

\begin{proposition} Let $R\ne 0$ be a commutative ring. Let $Q=\ter{R}$ be
defined on $R\times R\times R$ by
\begin{displaymath}
    (x_1,x_2,x_3)(y_1,y_2,y_3) = (x_1+y_1+(y_2+x_2)x_3y_3,x_2+y_2,x_3+y_3).
\end{displaymath}
Then $Q$ is a commutative A-loop satisfying $\lnuc{Q}=Z(Q)=R\times 0\times 0$
and $\mnuc{Q} = R\times R\times 0$.
\end{proposition}
\begin{proof}
See the relevant parts of the proof of Proposition \ref{Pr:Terg}.
\end{proof}

\subsection{Towards the classification of commutative A-loops of order $p^3$}

The results obtained up to this point come close to describing the isomorphism
types of all loops $\terg{\mathbb Z_p}{a}{b}$ for all primes $p\ne 3$.

Fix $p\ne 3$. The loop $\terg{\mathbb Z_p}{0}{0}$ is of exponent $p$ and is not
isomorphic to any other loop $\terg{\mathbb Z_p}{a}{b}$, by Lemma
\ref{Lm:IsoExp}. By Lemma \ref{Lm:IsoExp} and Corollary \ref{Cr:IsoExp}, the
loops $\{\terg{\mathbb Z_p}{0}{b};\;0<b<p\}$ form an isomorphism class. By
Lemmas \ref{Lm:IsoQuadraticResidue} and \ref{Lm:Isof}, each of the two sets
$I_r = \{\terg{\mathbb Z_p}{a}{b};\;a>0$ is a quadratic residue modulo $p$ and
$0\le b<p\}$ and $I_n = \{\terg{\mathbb Z_p}{a}{b};\;a>0$ is a quadratic
nonresidue modulo $p$ and $0\le b<p\}$ consist of pairwise isomorphic loops.

However, we did not manage to establish the following:

\begin{conjecture}\label{Cr:IsoQuadraticResidue}
Let $p>3$ be a prime, let $a_1\in\mathbb Z_p^*$ be a
quadratic residue and $a_2\in\mathbb Z_p^*$ be a quadratic nonresidue. Then
$\terg{\mathbb Z_p}{a_1}{0}$ is not isomorphic to $\terg{\mathbb Z_p}{a_2}{0}$.
\end{conjecture}

We have verified the conjecture computationally with the GAP \cite{GAP} package
LOOPS \cite{LOOPS} for $p=5$, $7$. It appears that one of the distinguishing
isomorphism invariants is the multiplication group $\mlt{Q} = \langle
L_x,\,R_x;\;x\in Q\rangle$.

The loops $\terg{\mathbb Z_p}{a}{b}$ behave differently for $p=3$ due to the
fact that $3$ is the only prime $p$ for which $p$ does not divide
$2\binom{p+1}{3}$. Denote by $f_{(A,B,C)}$ the bijection defined by
\eqref{Eq:Isof}. It can be verified by computer that $f_{(0,1,1)}$ is an
exceptional isomorphism $\terg{\mathbb Z_3}{0}{0}\to\terg{\mathbb Z_3}{0}{1}$,
$f_{(0,0,2)}$ is an isomorphism $\terg{\mathbb Z_3}{1}{1}\to\terg{\mathbb
Z_3}{1}{2}$, $f_{(0,1,2)}$ is an isomorphism $\terg{\mathbb
Z_3}{2}{0}\to\terg{\mathbb Z_3}{2}{1}$ and $f_{(0,1,1)}$ is an isomorphism
$\terg{\mathbb Z_3}{2}{0}\to\terg{\mathbb Z_3}{2}{2}$. The loops $\terg{\mathbb
Z_3}{0}{0}$, $\terg{\mathbb Z_3}{1}{0}$, $\terg{\mathbb Z_3}{1}{1}$ and
$\terg{\mathbb Z_3}{2}{0}$ contain precisely $12$, $6$, $24$ and $18$ elements
of order $9$, respectively, so no two of them are isomorphic.

\setlength{\unitlength}{1.0mm}
\begin{figure}\input{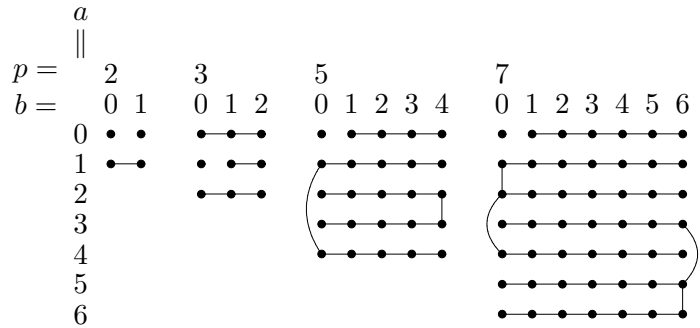}
\caption{Isomorphism classes of loops $\terg{\mathbb Z_p}{a}{b}$ for $p\in\{2,3,5,7\}$.}
\label{Fg:IsoTypes}
\end{figure}

Altogether, Figure \ref{Fg:IsoTypes} depicts the isomorphism classes of loops
$\terg{\mathbb Z_p}{a}{b}$ as connected components, for $p\in\{2,3,5,7\}$ and
$a$, $b\in\mathbb Z_p$. Moreover, if Conjecture \ref{Cr:IsoQuadraticResidue} is
true, the pattern established by $p=2$, $5$ and $7$ continues for all primes
$p>7$.

It is reasonable to ask whether, for an odd prime $p$, there are nonassociative
commutative A-loops of order $p^3$ not of the form $\terg{\mathbb Z_p}{a}{b}$.

Using a linear-algebraic approach to cocycles (see Subsection \ref{Ss:32}), we
managed to classify all nonassociative commutative A-loops of order $p^3$ with
nontrivial center, for $p\in\{2,3,5,7\}$. It turns out that all such loops are
of the type $\terg{\mathbb Z_p}{a}{b}$. In particular, $p=3$ is the only prime
for which there is no nonassociative commutative A-loop of order $p^3$ and
exponent $p$.

\begin{problem}
Let $p$ be an odd prime and $Q$ a nonassociative commutative A-loop of order
$p^3$. Is $Q$ isomorphic to $\terg{\mathbb Z_p}{a}{b}$ for some $a$,
$b\in\mathbb Z_p$?
\end{problem}

\section{Enumeration}\label{Sc:Enumeration}

We believe that future work will benefit from an enumeration of small
commutative A-loops. The results are summarized in Table \ref{Tb:Enum}, which
lists all orders $n\le 32$ for which there exists a nonassociative commutative
A-loop.

\begin{table}
\caption{Commutative A-loops up to isomorphism (up to
isotopism).}\label{Tb:Enum}
\begin{displaymath}
\begin{array}{r||c|c|c|c|c|c|c|c}
    \begin{array}{r}
        n
    \end{array}
        &8&15&16&21&24&27&30&32\\
    \hline\hline
    \begin{array}{r}
        \text{commutative groups}
    \end{array}
        &3&1&5&1&3&3&1&7\\
    \hline
    \begin{array}{r}
        \text{commutative nonassociative A-loops}
    \end{array}
        &4(3)&1&46(38)&1&4(3)&4&1&?\\
    \hline
    \begin{array}{r}
        \text{commutative nonassociative A-loops}\\
        \text{with nontrivial center}
    \end{array}
        &3(2)&0&44(37)&0&4(3)&4&1&?\\
    \hline
    \begin{array}{r}
        \text{commutative nonassociative A-loops}\\
        \text{of exponent $p$}
    \end{array}
        &2&-&12(11)&-&-&0&-&?\\
    \hline
    \begin{array}{r}
        \text{commutative nonassociative A-loops}\\
        \text{of exponent $p$}\\
        \text{with nontrivial center}
    \end{array}
        &1&-&10&-&-&0&-&211(210)
\end{array}
\end{displaymath}
\end{table}

If there is only one number in a cell of the table, it is both the number of
isomorphism classes and the number of isotopism classes. If there are two
numbers in a cell, the first one is the number of isomorphism classes and the
second one (in parentheses) is the number of isotopism classes.

All computations were done with the finite model builder Mace4 and with the GAP
package LOOPS on a Unix machine with a single $2$ GHz processor, with
computational times for individual orders ranging from seconds to hours.

\subsection{Comments on commutative A-loops of order $8$}

For classification up to isomorphism, see Section \ref{Sc:AppsIndex2}.

\begin{lemma}\label{Lm:GfItp}
Let $G$ be a commutative loop, $g\in \aut{G}$, and let $t_1$, $t_2$ be fixed
points of $g$. Define $f_i(x)=g(x)t_i$, for $i=1$, $2$. If there is $z\in G$
such that $g(z)=z^{-1}t_1^{-1}t_2$, then $G(f_1)$, $G(f_2)$ are isotopic.
\end{lemma}
\begin{proof}
Denote by $*$ the multiplication in $G(f_1)$ and by $\circ$ the multiplication
in $G(f_2)$. For $x\in G$, define $\alpha(x)=x$, $\alpha(\ov{x}) =
\ov{xz^{-1}}$, $\beta(x)=zx$, $\beta(\ov{x})=\ov{x}$, $\gamma(x)=zx$, and
$\gamma(\ov{x})=\ov{x}$. Then
\begin{align*}
    \alpha(x)\circ\beta(y) &= x\circ zy = xzy = \gamma(xy) = \gamma(x*y),\\
    \alpha(x)\circ\beta(\ov{y}) &= x\circ\ov{y} = \ov{xy} = \gamma(\ov{xy}) = \gamma(x*\ov{y}),\\
    \alpha(\ov{x})\circ\beta(y) &= \ov{xz^{-1}}\circ zy = \ov{xy} = \gamma(\ov{xy}) = \gamma(\ov{x}*y),\\
    \alpha(\ov{x})\circ\beta(\ov{y}) &= \ov{xz^{-1}}\circ \ov{y} = g(xz^{-1}y)t_2 = zg(xy)t_1
    = \gamma(g(xy)t_1) = \gamma(\ov{x}*\ov{y}),
\end{align*}
where we have used $g(z)=z^{-1}t_1^{-1}t_2$ in the last line.
\end{proof}

Let $G=\mathbb Z_2\times\mathbb Z_2 = \langle a\rangle \times \langle b\rangle$
be the Klein group. Consider the transposition $g=(a,b)$ with fixed points
$t_1=1$, $t_2=ab$. Let $f_i(x) = g(x)t_i$, for $i=1$, $2$. Then $b = g(a) =
a^{-1}t_1^{-1}t_2$, so $G(f_1)$, $G(f_2)$ are isotopic by Lemma \ref{Lm:GfItp}.

\subsection{Comments on commutative A-loops of order $15$ and $21$}

\begin{lemma}\label{Lm:pq}
Let $Q$ be a nonassociative commutative A-loop of order $p_0p_1$, where $p_0\ne
p_1$ are odd primes. Then there is $0\le i\le 1$ such that $Q$ contains a
normal subloop $S$ of order $p_i$, and all elements in $Q\setminus S$ have
order $p_{i+1}$, where the subscript is calculated modulo $2$.
\end{lemma}
\begin{proof}
We will use results of \cite{JKV} mentioned in the introduction without further
reference. Since $Q$ is of odd order, it is solvable. Since $Q$ is also
nonassociative, there is a normal subloop $S$ of $Q$ such that $1\ne S\ne Q$.
By the Lagrange Theorem, $|S|=p_i$ for some $0\le i\le 1$. Without loss of
generality, let $|S|=p_0$. Let $y\in Q\setminus S$ and let $T$ be the preimage
of the subloop $\langle yS\rangle$ of $Q/S$. By the Lagrange Theorem again,
$y^{p_1}=1$, as the only other alternative $|y|=p_0p_1$ would mean that $Q$ is
a group by power-associativity.
\end{proof}

The information afforded by Lemma \ref{Lm:pq} is sufficient to construct all
nonassociative commutative A-loops of order $15$ and $21$ by the finite model
builder Mace4. It turns out that in each case there is a unique such loop.
These two loops were constructed already by Dr\'apal \cite[Proposition
3.1]{Drapal}. Nevertheless the following problem remains open:

\begin{problem} Classify commutative A-loops of order $pq$, where $p<q$ are
odd primes.
\end{problem}

We have some reasons to believe that there is no nonassociative commutative
A-loop of order $35$.

\subsection{Comments on commutative A-loops of order $16$}\label{Ss:16}

Among the $12$ nonassociative commutative A-loops of order $16$ and exponent
$2$, three have inner mapping groups of orders that are not a power of $2$,
namely $12$, $56$ and $56$. We now construct the two nonassociative commutative
A-loops of order $16$ and exponent $2$ with inner mapping groups of order $56$,
and we show that they are isotopic.

Let $G=\mathbb Z_4\times\mathbb Z_2$. Define $g\in\aut{G}$ by $g(i,j) =
(i,i+j\mod 2)$. Note that $t_1=(0,0)$, $t_2=(2,1)$ are fixed points of $g$, and
let $f_i(x)=g(x)+t_i$. Then $G(f_1)$, $G(f_2)$ are the two announced loops, and
they are isotopic by Lemma \ref{Lm:GfItp}, since $g(1,0)=(1,1)$ and
$-(1,0)-(0,0)+(2,1)=(1,1)$.

\subsection{Comments on commutative A-loops of order $32$ and exponent $2$ with nontrivial
center}\label{Ss:32}

The methods developed in \cite{NV} in order to classify Moufang loops of order
$64$ can be adopted to other classes of loops. Using the cocycle formula of
Corollary \ref{Cr:CentralExtension} and the classification of commutative
A-loops of order $16$ from Subsection \ref{Ss:16}, we were able to classify all
commutative A-loops of order $32$ and of exponent $2$ with nontrivial center.

We now briefly describe the search, following the method of \cite{NV} closely.
For more details, see \cite{NV}.

Let $Q$ be a commutative A-loop of order $32$ and exponent $2$ with nontrivial
center. Then $Z(Q)$ is obviously an elementary abelian $2$-group, and hence it
possesses a $2$-element central subgroup $Z=(Z,+,0)$. Then $Q/Z=K$ is a
commutative A-loop of order $16$ and exponent $2$.

The loop cocycles $\theta:K\times K\to Z$ form a vector space $V$ over
$Z=\gf{2}$ with respect to addition $(\theta+\mu)(x,y) = \theta(x,y)+\mu(x,y)$.
The vector space $V$ has basis $\{\theta_{u,v};\;1\ne u\in K,\,1\ne v\in K\}$,
where
\begin{displaymath}
    \theta_{u,v}(x,y) = \left\{\begin{array}{ll}
        1,&\text{if $(u,v)=(x,y)$},\\
        0,&\text{otherwise.}
    \end{array}\right.
\end{displaymath}
The extension $K\ltimes_\theta Z$ will be a commutative A-loop of exponent $2$
if and only if $\theta$ belongs to the subspace $C = \{\theta\in V;\;\theta$
satisfies \eqref{Eq:Cocycle}, $\theta(x,x)=0$ for every $x\in K$ and
$\theta(x,y)=\theta(y,x)$ for every $x$, $y\in K\}$.

For every $x$, $y$, $z$, $x'\in K$, the equation \eqref{Eq:Cocycle} can be
viewed as a linear equation over $\gf{2}$ in variables $\theta_{u,v}$.
Similarly, for every $x$, $y\in K$ we obtain linear equations from the
condition $\theta(x,y)=\theta(y,x)$, and from $\theta(x,x)=0$.

Upon solving this system of linear equations, we obtain (a basis of) $C$, and
it is in principle possible to construct all extensions $K\ltimes_\theta Z$ for
$\theta\in C$. The two computational problems we face are: (i) the dimension of
$C$ can be large, (ii) it is costly to sort the resulting loops up to
isomorphism. In order to overcome these obstacles, we take advantage of
coboundaries and of an induced action of $\aut{K}$ on $C$.

Let $\tau:K\times Z$ be a mapping satisfying $\tau(1)=0$. Then
$\delta\tau:K\times K\to Z$ defined by
\begin{displaymath}
    \delta\tau(x,y) = \tau(xy)-\tau(x)-\tau(y)
\end{displaymath}
is a \emph{coboundary}. Coboundaries form a subspace $B$ of $V$.

In fact, $B$ is a subspace of $C$. This can be proved explicitly by verifying
that every coboundary $\theta = \delta\tau$ satisfies the identity
\eqref{Eq:Cocycle}, $\theta(x,y)=\theta(y,x)$ and $\theta(x,x)=0$. The
verification of \eqref{Eq:Cocycle} is a bit unpleasant, so it is worth
realizing that every coboundary $\theta$ satisfies the group cocycle identity
\begin{displaymath}
    \theta(x,y)+\theta(xy,z) = \theta(y,x)+\theta(x,yz),
\end{displaymath}
and hence also any cocycle identity that follows from associativity, in
particular \eqref{Eq:Cocycle}.

Moreover, if $\theta$, $\mu:K\times K\to Z$ are two cocycles such that
$\theta-\mu$ is a coboundary, then $K\ltimes_\theta Z$ is isomorphic to
$K\ltimes_\theta Z$, cf. \cite[Lemma 9]{NV}. It therefore suffices to construct
loops $K\ltimes_\theta Z$, where $\theta\in D$, $C=B\oplus D$.

Given $\theta\in V$ and $\varphi\in\aut{K}$, we define $\theta_\varphi\in V$ by
\begin{displaymath}
    \theta_\varphi(x,y) = \theta(\varphi(x),\varphi(y)).
\end{displaymath}
This action of $\aut{K}$ on $V$ induces an action on $D$. Moreover, by
\cite[Lemma 14]{NV}, $K\ltimes_\theta Z$ is isomorphic to
$K\ltimes_{\theta_\varphi}Z$. It therefore suffices to construct loops
$K\ltimes_\theta Z$, where we take one $\theta$ from each orbit of $\aut{K}$ on
$D$.

Using each of the $13$ commutative A-loops of order $16$ and exponent $2$ as
$K$ (the elementary abelian group of order $16$ must also be taken into
account), the above search finds $355$ commutative A-loops of order $32$ and
exponent $2$ within several minutes. The final isomorphism search takes several
hours with LOOPS.

The lone isotopism $\mathbb Z_2\times Q_1 \to \mathbb Z_2\times Q_2$ is induced
by the isotopism $Q_1\to Q_2$ described in Subsection \ref{Ss:16}.

\section{Acknowledgement}

We thank the anonymous referee for the nice proof of Lemma \ref{Lm:NewForms}.

\bibliographystyle{plain}

\end{document}